\documentclass[11pt]{amsart}

\usepackage{amssymb,xspace}

\usepackage{hyperref}

\newcommand*{\C}{\ensuremath{\mathbb C}\xspace} 
\newcommand*{\hatC}{\ensuremath{\hat C}\xspace} 

\newcommand*\F{\ensuremath{\mathcal F}\xspace}

\newcommand*\Ooo{\ensuremath{\mathcal O}\xspace}
\newcommand*{\PP}{\ensuremath{\mathbb P}\xspace}

\newcommand*{\R}{\ensuremath{\mathbb R}\xspace}
\newcommand*{\T}{\ensuremath{\mathcal T}\xspace}
\newcommand*{\X}{\ensuremath{\mathcal X}\xspace}
\newcommand*{\Z}{\ensuremath{\mathbb Z}\xspace}
\newcommand*{\Zcal}{\ensuremath{\mathcal Z}\xspace}
\def \lmod#1\rmod {\vphantom{#1}\left|\smash{#1}\right|}

\DeclareMathOperator{\Aut}{Aut}
\DeclareMathOperator{\Perm}{Perm}

\DeclareMathOperator{\Spec}{Spec}
\DeclareMathOperator{\Sym}{Sym}

\DeclareMathOperator{\pr}{pr}

\let\ph\varphi

\newtheorem{proposition}{Proposition}[section]
\newtheorem{theorem}[proposition]{Theorem}

\newtheorem{lemma}[proposition]{Lemma}
\newtheorem{corollary}[proposition]{Corollary}

\theoremstyle{remark}
\newtheorem{note}[proposition]{Remark}

\theoremstyle{definition}
\newtheorem{definition}[proposition]{Definition}

\title[On projections of smooth and nodal plane curves]{On projections of 
smooth and nodal plane curves}

\author{Yu. Burman}

\address{Independent University of Moscow, 11, B.Vlassievsky per., Moscow, 
Russia, 119002 and National Research University Higher School of Economics, 7 Vavilova str., Moscow, 
Russia, 117312} 

\email{yburman@gmail.com}

\author{Serge Lvovski}

\address{National Research University Higher School of Economics and
AG Laboratory HSE, 7 Vavilova str., Moscow, Russia, 117312}

\email{lvovski@gmail.com}

\thanks{Y.B.\ was supported by the Lady Davis foundation visiting 
professorship grant, HSE Scientific Foundation individual research grant 
12-01-0015, the MON grant 02.740.11.0608, the Simons--IUM fellowship by the 
Simons foundation, by the 2013 Dobrushin professorship grant and by the 
RFBR grant N.Sh.5138.2014.1 (``Scientific school of V.I.Arnold''). 
S.L.\ was partially supported by AG Laboratory NRU HSE, RF government 
grant, ag.\ 11.G34.31.0023}



\begin{document}

 \begin{abstract}
Suppose that $C\subset\PP^2$ is a general enough nodal plane curve of
degree $>2$, $\nu\colon \hat C\to C$ is its normalization, and 
$\pi\colon \hat C\to\PP^1$ is a finite morphism simply ramified over
the same set of points as a projection $\pr_p\circ\nu\colon\hat C
\to\PP^1$, where $p\in\PP^2\setminus C$ (if $\deg C=3$, one should
assume in addition that $\deg\pi\ne4$). We prove that the morphism
$\pi$ is equivalent to such a projection if and only if it extends to
a finite morphism $X\to(\PP^2)^*$ ramified over $C^*$, where $X$ is a
smooth surface.

As a by-product, we prove the Chisini conjecture for mappings ramified
over duals to general nodal curves of any degree~$\ge3$ except for
duals to smooth cubics; this strengthens one of Victor Kulikov's
results.
 \end{abstract}

\maketitle

\section{Introduction}

Let $C \subset \PP^2$ be a projective curve of degree $d > 2$ such that all 
its singularities are nodes, and let $\nu\colon \hatC \to C$ be the 
normalization mapping. For a point $p \notin C$ consider the projection 
$\pr_p: C \to \PP^1 = p^\perp$ from $p$; here $p^\perp \subset \mathbb P^2$ 
is the set of lines passing through $p$, and $\pr _p$ assigns to $x$ the 
line joining $p$ and $x$. The composition $\pr _p \circ \nu: \hat C \to 
\mathbb P^1$ will be called a generalized projection; its branch locus 
coincides with $p^\perp \cap C^*$, where $C^* \subset (\PP^2)^*$ is the 
curve dual to $C$. The generalized projection has simple ramification if 
and only if $p^\perp$ is transversal to $C^*$.

Suppose that $\pi\colon C'\to p^\perp$, where $C'$ is a smooth projective 
curve, is a holomorphic mapping with simple ramification and such that the 
branch locus of $\pi$ coincides with $p^\perp \cap C^*$. We are looking for 
a criterion for $\pi$ to be isomorphic to the generalized projection 
$\pr_p\circ\nu$.

\subsubsection*{Problem background}
Consider first the case when $C$ is smooth. An easy dimension count
shows that if $d > 3$ then the branch locus of a projection is not
arbitrary.  Namely, it follows from the Riemann--Hurwitz formula that
a degree $d$ map from a curve of degree $d$ has $d(d-1)$ critical
values, provided the ramification is simple. So, its branch locus is a
point of $\Sym^{d(d-1)}\PP^1 = \PP^{d(d-1)}$. The space of projective
curves $C \subset \PP^2$ of degree $d$ has dimension $\binom{d+2}{2}-1
= d(d+3)/2$.  For a point $p \in \PP^2$ there is a $3$-dimensional
group of projective automorphisms $\PP^2 \to \PP^2$
preserving $p$ and all the lines containing $p$. So, the space of all
projections has the dimension $m_d \mathrel{:=} d(d+3)/2-3 =
(d^2+3d-6)/2 \ll d(d-1)$ for large $d$.

Consider now small $d$. For $d=1$ and $d=2$ the situation is trivial. For 
$d=3$ one has $m_3 = 6 = 3(3-1)$, and indeed, it is easy to see that any 
$6$ pairwise distinct points $a_1, \dots, a_6 \in \PP^1$ are branch points 
of a suitable generic projection of a cubic curve.

The case $d=4$ was studied in detail by R.Vakil in \cite{Vakil}. Here
$m_4 = 11$, $4(4-1) = 12$, so branch loci of generic projections lie
in a hypersurface $V \subset \PP^{12}$. For a point $a \in V$ there
are $255$ pairs $(C',\pi)$ where $C'$ is a smooth curve and $a$ is the
branch locus of the mapping $\pi: C' \to \PP^1$ of degree $4$. For
$120$ pairs $C'$ is a smooth plane projective curve and $\pi$ is
isomorphic to a projection. For the remaining $135$ pairs the curve
$C'$ is genus $3$ hyperelliptic, hence not plane.

For $d > 4$ it is easy to derive from~\cite[Proposition 1]{EGH} that if 
$C'$ is a smooth plane projective curve then each degree $d$ mapping $\pi: 
C' \to \PP^1$ is isomorphic to a projection. We do not, though, assume $C'$ 
to be a plane curve, so a situation similar to the $d=4$ 
case may arise; cf.\ Proposition~\ref{prop:example}. 

For $C$ nodal a degree $d$ map $\hatC \to \PP^1$ does not need to be 
equivalent to a projection.

\subsubsection*{Main results} 
To make our problem more tractable we impose some generality
hypotheses on the curves and mappings in question. To wit, we will
assume that the mappings $\pi\colon\hatC\to\PP^1$ have simple
ramification (see Definition~\ref{def:moderate_covering} below) and
that the nodal curve $C\subset\PP^2$ is general enough in the sense of
Definition~\ref{def:general_curve}. If $\nu\colon\hatC\to C$ is the
normalization, then for the general point $p\in\PP^2\setminus C$ the
generalized projection $\pr _p\circ\nu \colon \hatC\to\PP^1$ has
simple ramification.

We will say that the holomorphic mappings $\ph_1\colon C_1\to\PP^1$ and 
$\ph_2\colon C_2\to\PP^1$ are \emph{equivalent} if there exists a 
holomorphic isomorphism $\psi\colon C_1\to C_2$ such that 
$\ph_2\circ\psi=\ph_1$.

 \begin{theorem}\label{th.main}
Suppose that $C\subset \PP^2$ is a nodal curve of degree~$>2$ that is 
general enough in the sense of Definition~\ref{def:general_curve}. Suppose 
that a point $p\in\PP^2\setminus C$ is such that the composition 
$\pr_p\circ\nu\colon\hatC\to\PP^1$, where $\nu\colon\hatC\to C$ is the 
normalization and $\pr_p\colon C \to p^\perp = \PP^1$ is the projection 
from $p$, has simple ramification. Suppose that $C'$ is a smooth projective 
curve and $\pi\colon C'\to p^\perp$ is a holomorphic mapping with simple 
ramification such that the branch locus of $\pi$ coincides with $p^\perp 
\cap C^*$. If $\deg C=3$, assume in addition that $C$ has a node or 
$\deg\pi\ne4$. Then the following two conditions are equivalent:

\textup{(a)} $\pi$ is equivalent to $\pr _p \circ \nu$: there exists an 
isomorphism $\ph\colon C'\to\hat C$ such that $(\pr_p\circ\nu)\circ 
\ph=\pi$.

\textup{(b)} There exist a smooth projective surface $X$, a finite 
holomorphic mapping $f\colon X\to(\PP^2)^*$ that is ramified exactly over 
$C^*$, and an isomorphism $\ph\colon C'\to f^{-1}(p^\perp)$ such that 
$\left.f\right|_{f^{-1}(p^\perp)}\circ \ph=\pi$. 
 \end{theorem} 

 \begin{note}
The implication $(\mathrm a)\Rightarrow(\mathrm b)$ holds without the 
genericity hypotheses, for arbitrary nodal curves and arbitrary generalized 
projections. See the proof of Proposition~\ref{gen(a)=>(b)} below.
 \end{note}

\begin{note}
If $C$ is a smooth cubic and $\deg f = 4$ then the theorem is wrong; 
see Remark~\ref{chisini:example}.
\end{note}

\begin{note}
The condition of smoothness of $X$ in the theorem cannot be 
omitted, at least for $\deg f$ equal to $3$ or $4$: in Section 
\ref{sec:example} we show that for $d=3$ and $4$ there exists 
a general enough smooth curve $C\subset\PP^2$ of degree $d$, a line 
$p^\perp \subset (\PP^2)^*$ (for some $p \in \PP^2$) transversal to~$C^*$, 
a smooth projective curve $C'$, and a morphism $\pi\colon C'\to p^\perp$ 
simply ramified over $p^\perp \cap C^*$ such that condition~(b) of Theorem 
\ref{th.main} with the smoothness omitted holds but $\pi$ is not equivalent 
to the projection $\pr _p$ (Proposition~\ref{prop:example}). We do not know 
similar counterexamples with $d > 4$.
\end{note}

In other words, if we are given a simply ramified covering of $\PP^1$
with the same branch locus as that of a (generalized) projection from
a given point~$p$, then this covering is the generalized projection if
and only if it can be extended to a finite ramified covering of
$(\PP^2)^*$ branched over $C^*$, with a smooth surface upstairs.

Theorem~\ref{th.main} may be restated in topological terms: we will show 
that if $L \subset (\PP^2)^*$ is a projective line then the mapping 
$\pi\colon C'\to L$ is equivalent to the generalized projection 
$\pr_p\circ\nu$ if and only if the covering $\pi^{-1}(L\setminus C^*)\to 
L\setminus C^*$ can be extended to a covering $X_0\to(\PP^2)^*\setminus 
C^*$ such that its fiber monodromy satisfies some extra conditions ($C^*$ 
must have no bad nodes or bad cusps in the sense of 
Definitions~\ref{def:good.nodes} and~\ref{def:bad.cusp}); see 
Theorem~\ref{th.main2}.

Note that we make (almost) no assumption about genus of the curve $C'$ or 
degree of the morphism~$f$.

The $(\mathrm a) \Rightarrow (\mathrm b)$ implication in
Theorem~\ref{th.main} is easy (see Section~\ref{(a)=>(b)}). If $\deg C\ge 
7$, the $(\mathrm b) \Rightarrow (\mathrm a)$ implication 
follows very easily from Theorem~10 of Victor Kulikov's 
paper~\cite{Kulikov1999} (see Proposition~\ref{Chisini_argument} below), 
but for the remaining cases $3\le\deg C\leq 6$, where the Chisini 
conjecture for coverings ramified over $C^*$ is not completely proved, it 
requires more work. We give an argument that works uniformly in all cases, 
including the case $\deg C\ge 7$; our proof is
quite different from Kulikov's one. As a by-product, we 
establish the Chisini conjecture for coverings of $\PP^2$ ramified over 
curves dual to general enough nodal curves of degrees $4$, $5$, and $6$ 
(and of degree $3$ provided that either the curve does have a node or 
degree of the covering is not $4$). This extends Theorem~10 
from~\cite{Kulikov1999}.

The paper is organized as follows. In Section~\ref{(a)=>(b)} we prove the 
easy part of Theorem~\ref{th.main}. In Section~\ref{sec:plastering} we 
study coverings of the projective plane that are ramified over a curve 
having only nodes and standard cusps as singularities, and in 
Section~\ref{sec:PL} we prove the difficult part of Theorem~\ref{th.main}. 
Finally, in Section~\ref{sec:example} we show that, at least if $\deg C = 
3$ and $\deg C = 4$, the smoothness condition in Theorem~\ref{th.main} 
cannot be omitted.

The authors would like to thank Maxim Kazarian, Victor Kulikov, Alexander 
Kuznetsov, Sergey Lando, Sergey Natanzon, Ossip Schwarzman, Andrey 
Soldatenkov, and Victor Vassiliev for useful discussions. We are grateful 
to Igor Khavkine, Francesco Polizzi, Damian R\"ossler, Will Sawin, and 
particularly Rita Pardini for consultations at 
\url{http://mathoverflow.net}. 

The first named author wishes to thank the Technion --- Israel Institute of 
Technology, where the final version of the paper was written, for the warm 
hospitality.

Our special thanks go to Boris Shapiro, at whose instigation this research 
was started. Without his constant encouragement this paper would never have 
been written.

\section*{Notation, conventions, and definitions}

The base field for all the algebraic varieties will be the field \C of
complex numbers.

Except for the case when we mention explicitly Zariski topology, all
topological spaces will be assumed Hausdorff and locally simply
connected, and all the topological terms will refer to the classical
topology of complex algebraic varieties.

We will say that a singular point of a curve is a \emph{node} if it is
locally analytically isomorphic to the singularity of the plane curve
defined by the equation $xy=0$, and that a singular point is a
\emph{standard cusp} (or simply a cusp) if it is locally analytically
isomorphic to the singularity of the plane curve defined by the
equation $y^2+x^3=0$.

If $f\colon X\to Y$ is a finite holomorphic map of smooth varieties of
equal dimension, then by \emph{ramification locus} (a.k.a.\ the set of
critical points) of $f$ we mean the closed subset
 \[ 
R=\{x\in X\mid\text{derivative of $f$ is 
degenerate at $x$}\}; 
 \] 
by \emph{branch locus} (a.k.a.\ the set of critical values) of $f$ we mean 
the subset $f(R)\subset Y$.

Let $f\colon W\to V$ be a covering of topological spaces, and take $v_0\in 
V$.  Then the action of the group $\pi_1(V,v_0)$ on the set $f^{-1}(v_0)$ 
via loop lifting will be called the \emph{fiber monodromy} of $f$.

If $f\colon Y \to X$ is a finite morphism of algebraic varieties of equal 
dimension (i.e., a proper holomorphic map with finite fibers) and $B 
\subset X$ is its branch locus, then the restriction of $f$ to $f^{-1}(X 
\setminus B) \subset Y$ is a finite-sheeted covering of topological spaces. 
Take $x_0 \in X \setminus B$. Then by fiber monodromy of $f$ we will mean 
fiber monodromy of the covering $f^{-1}(X \setminus B) \to Y\setminus B$ 
with respect to the point~$x_0$.

Suppose now that $X = \PP^1$, so the branch locus $B\subset \PP^1$ is a 
finite set. If $p \in B$ and $\gamma_p$ is the conjugacy class in 
$\pi_1(\PP^1\setminus B)$ corresponding to a small loop around $p$, then 
the image of $\gamma_p$ under the fiber monodromy homomorphism is a 
conjugacy class in the symmetric group~$S_d$, where $d$ is the number of 
points in $f^{-1}(x_0)$. This conjugacy class (or the corresponding 
partition of $d$) is called the \emph{cyclic type} of the point~$p$.

 \begin{definition}\label{def:moderate_covering}
Suppose that $C$ is a smooth projective curve. We will say that a finite 
morphism $f\colon C\to\PP^1$ has \emph{simple ramification} if each branch 
point $\xi\in\PP^1$ has the cyclic type of a transposition. 
 \end{definition} 

If $C \subset \PP^2$ is a nodal curve and $\nu\colon\hatC\to C$ is the 
normalization, then the generalized projection of $\hatC$ from a general 
point of $\PP^2$ has simple ramification.

 \begin{definition}\label{def:general_curve}
Suppose that $C\subset\PP^2$ is an irreducible projective curve such that 
each singular point of $C$ is a node. Let us say that $C$ is \emph{general 
enough} if the following conditions are satisfied.

-- all the inflexion points of $C$ are simple;

-- no line is tangent to $C$ at more than two points;

-- if a line is tangent to $C$ at an inflexion point, it is not tangent to 
$C$ elsewhere.
 \end{definition}

Here, we assume that a line is tangent to $C$ if it is either tangent to 
$C$ at a smooth point or tangent to a branch of $C$ at a node. By inflexion 
point we mean either a smooth point $x\in C$ for which the local 
intersection index of $C$ with its tangent at $x$ is greater than $2$, or a 
node $x\in C$ for which the local intersection index of at least one of the 
two (limiting) tangent lines to branches a $x$ with the corresponding 
branch is greater than~$2$; we say that an inflexion point is \emph{simple} 
if the intersection index in question equals exactly~$3$.

For a projective subspace $\alpha \subset \PP^n$ we denote by $\alpha^\perp 
\subset (\PP^n)^*$ the set of hyperplanes in $\PP^n$ containing $\alpha$, 
where $(\PP^n)^*$ is the dual projective space. If $\dim \alpha = k$ then 
$\alpha^\perp \subset (\PP^n)^*$ is a projective subspace of dimension 
$n-1-k$; in particular, if $\alpha$ is a point then $\alpha^\perp$ is a 
hyperplane (denoted by $H_\alpha$ in SGA7, see \cite{SGA7.2}).

If $X \subset \PP^n$ is a projective variety, then by 
$X^*\subset{(\PP^n)}^*$ we will denote its projective dual, i.e., the 
closure of the set of hyperplanes tangent to $X$ at smooth points.

We also will be using the notation $\alpha^\perp$ and $X^*$ when $\alpha, X 
\subset (\PP^n)^*$, where the canonical isomorphism $((\PP^n)^*)^* = \PP^n$ 
is assumed.

If $X\subset\PP^2$ is a projective curve, we say that a line 
$L\subset\PP^2$ is transversal to $X$ if $L$ does not pass through singular 
points of $X$ and is not tangent to $X$ at any non-singular point.

If $D_1$ and $D_2$ are Cartier divisors on a projective surface, then 
$(D_1,D_2)$ stands for their intersection index. 

\section{Proof of the $(\mathrm a)
      \Rightarrow (\mathrm b)$ implication in Theorem
      \ref{th.main}}\label{(a)=>(b)}

This part of Theorem~\ref{th.main} follows immediately from

 \begin{proposition}\label{gen(a)=>(b)}
Suppose that $C\subset\PP^2$ is a nodal curve, $\nu\colon \hat C\to C$
is the normalization, $p\in\PP^2\setminus C$, and $\pr _p\colon C\to
p^\perp$ is the projection from $p$. Then there exists a smooth
projective surface $X_C$, a finite regular mapping $f_C\colon
X_C\to(\PP^2)^*$ having the dual curve $C^* \subset(\PP^2)^*$ as its
branch locus, and an isomorphism $\ph\colon \hatC\to
f_C^{-1}(p^\perp)$ such that $f_C\circ \ph=\pr_p\circ\nu$.
 \end{proposition}

 \begin{proof}
Put 
 \begin{equation}\label{def_of_X_C} 
X_C=\{(x,t)\in \hatC\times(\PP^2)^*\mid \nu(x)\in t^\perp\}.
 \end{equation} 
The surface $X_C$, being the projectivization of the vector bundle 
$\nu^*\T_{\PP^2}(-1)$, is smooth.

Define the mapping $f_C\colon X_C\to (\PP^2)^*$ by the formula
 \begin{equation}\label{def:f_C}
f(x,t) = t.
 \end{equation}
For $t\in(\PP^2)^*$,
 \[
f_C^{-1}(t)=\{(x,t)\mid \nu(x)\in t^\perp\};
 \]
$f$ is a finite morphism of degree~$d$ since for the general 
$t\in(\PP^2)^*$ the line $t^\perp$ intersects $C$ at $d$ smooth points. Its 
branch locus is the set of $t$ such that $t^\perp$ is either tangent to $C$ 
at a smooth point or tangent to a branch of $C$ at a node, so the branch 
locus of $f$ coincides with~$C^*$.

If $p\in \PP^2\setminus C$ then 
 \[
f_C^{-1}(p^\perp)=\{(x,t)\in \hatC\times(\PP^2)^*\mid 
\nu(x), p \in t^\perp\}.
 \]
It is clear that the mapping $\ph\colon \hatC\to f^{-1}(p^\perp)$ sending 
$x$ to the pair $(x,\overline{px})$ is the required isomorphism. 
 \end{proof}

\section{Plastering of generic coverings over complements 
to curves}\label{sec:plastering} 

In this section we prepare for the proof of the implication 
$\mathrm{(b)}\Rightarrow\mathrm{(a)}$ in Theorem~\ref{th.main}.

We start with a simple general fact. Suppose that $D\subset\PP^2$ is an 
arbitrary projective curve, $X_0$ is a smooth affine surface and $f_0\colon 
X_0\to \PP^2\setminus D$ is a finite (topological) covering.

 \begin{proposition}\label{prop:same_mono}
Suppose that, in the above setting, there exists a line
$L_0\subset\PP^2$ transversal to $D$ such that the induced covering
$f_0^{-1}(L_0)\to L_0\setminus D$ has simple ramification. Then for
any line $L\subset\PP^2$ transversal to $D$ the induced covering
$f_0^{-1}(L)\to L\setminus D$ also has simple ramification.
 \end{proposition}

Following \cite{Kulikov1999}, we will say that coverings satisfying
hypotheses of the proposition are \emph{generic}.

 \begin{proof}[Proof of the proposition]
The set 
 \[
U=\{t\in(\PP^2)^*\mid \text{$t^\perp$ is transversal to~$D$}\}
 \]
is arcwise connected; let $\gamma: [0,1] \to U$ be a path such that 
$\gamma(0) = L_0$, $\gamma(1) = L$. Let also 
 \[
\X=\{(x,t)\in(\PP^2\setminus D)\times U\mid x\in t^\perp\};
 \]
the natural projection $p\colon \X\to U$ is a locally trivial fiber bundle 
with the fiber homeomorphic to $\PP^1$ punctured at $\deg D$ points. The 
pullback covering $p^* f_0\colon \tilde \X\to \X$ restricted to $p^{-1}(t) 
\subset \X$ is isomorphic to the covering~$f_0^{-1}(t^\perp)\to t^\perp$. 
The pullback bundle $\gamma^* p: \overline X \to [0,1]$ is trivial, so the 
proposition follows. 
 \end{proof} 

If $f_0\colon X_0\to \PP^2\setminus D$ is a covering, then $X_0$
carries a unique structure of complex variety for which the mapping
$f$ is a finite unramified morphism. Suppose now that the covering is
generic and all the singularities of the curve $D$ are nodes or
standard cusps. It follows from a theorem of Grauert
and Remmert~\cite[Expos\'e~XII, Th\'eor\`eme 5.4]{SGA1} that in this
situation $f_0$ extends uniquely to a finite mapping $f\colon
X\to\PP^2$ with normal~$X$. A well-known GAGA-style
result~\cite[Expos\'e~12, Corollaire 4.6]{SGA1} implies that the
resulting complex space $X$ will be a projective
surface. We are going now to describe explicitly the singularities of
surface $X$ as well as the structure of the mapping $f$ near
ramification locus.

In the sequel, $\Delta$ will denote the unit disk $\{z\colon 
\left|z\right|<1\} \subset \C$, and $\Delta^*$ will stand for the punctured 
disk $\Delta \setminus \{0\}$.

Suppose that $q$ is a node of $D$; choose a neighborhood $U\ni q$ and an 
isomorphism $U\to \Delta^2$ such that $D\cap U$ is mapped to the set 
$\{(x,y) \in D^2\mid xy=0\}$. Fix a point $q_0\in U\setminus D$. The 
covering $f_0^{-1}(U\setminus D)\to U\setminus D$ induces a fiber monodromy 
action of the group $\pi_1(U\setminus D,q_0) = \Z \oplus \Z$ on the set 
$f^{-1}(q_0)$.

The two generators of $\pi_1(U\setminus D,q_0)$ are the small loops
around two branches of $D$ at $q$; since $f_0$ is a generic covering,
Proposition~\ref{prop:same_mono} shows that the generators act on
$f^{-1}(q_0)$ by transpositions. These transpositions commute, so they
can be either disjoint or equal.

 \begin{definition}\label{def:good.nodes}
In the above setting, we will say that a node $q\in D$ is \emph{good} (with 
respect to the covering $f_0$) if the two generators of 
$\pi_1(U\setminus D,q_0)$ act by disjoint transpositions.

We will say that a node $q\in D$ is \emph{bad} (with respect to $f_0$) if 
they act by equal transpositions.
 \end{definition}

Suppose now that $q$ is a (standard) cusp of the curve $D$. Choose a 
neighborhood $U\ni q$ and an isomorphism $U\to \Delta^2$ such that $D\cap 
U$ is mapped to the set $\{(x,y)\in D^2\mid x^3+y^2=0\}$. Fix a point 
$q_0\in U\setminus D$. The covering $X\setminus f^{-1}(D)\to\PP^2\setminus 
D$ induces a fiber monodromy action of the group $\pi_1(U\setminus D,q_0)$ 
on the set $f^{-1}(q_0)$. It is well known that $\pi_1(U\setminus D,q_0)$ 
is isomorphic to the Artin braid group on $3$ strings $B_3=\langle u,v\mid 
uvu=vuv\rangle$, where the generators $u$ and $v$ correspond to the two 
loops around two intersection points $\ell\cap D$, where $\ell$ is a line 
close to $q$.

Since $f_0$ is a generic covering, the elements $u$ and 
$v$ of $\pi_1(U\setminus D,q_0)$ act on $f^{-1}(q_0)$ by transpositions. 
These transpositions (call them $U$ and $V$) satisfy the relations $UVU = 
VUV$, so they are either non-commuting or equal.

 \begin{definition}\label{def:bad.cusp} 
In the above setting, we will say that the cusp $q\in D$ is \emph{good} 
(with respect to the covering $f_0$) if the transpositions $U$ and $V$ are 
non-commuting.

We will say that the cusp $q\in D$ is \emph{bad} (with respect to $f_0$) if 
$U = V$.
 \end{definition}

\begin{proposition}[cf.\ 
{\cite[Section 1.3]{Kulikov1999}}]\label{not.smooth.over.bad}
The local behavior of the mapping $f$ near points of the preimage 
$f^{-1}(q)$, $q \in D$, is as follows:
 \begin{enumerate}
\item\label{It:Smooth} If $q$ is a smooth point of $D$ then
  $f^{-1}(q)$ consists of $d-1$ points, and $X$ smooth
  at all of them. At exactly one of them $f$ is
  ramified, and near this point $f$ is locally isomorphic to $f(x,y) =
  (x,y^2)$.

\item\label{It:GoodNode} If $q$ is a good node then $f^{-1}(q)$ consists of 
$d-2$ points, and $X$ is smooth at all of them. At exactly two of them $f$ 
is ramified, with the same local behavior as in Case \ref{It:Smooth}. 

\item\label{It:BadNode} If $q$ is a bad node then $f^{-1}(q)$ consists of 
$d-1$ points. All of them except one are smooth. The remaining point is 
singular, locally isomorphic to $\{(x,y,z) \in \Delta^3 \mid z^2 = xy\}$ 
\textup(Du Val's $A_1$\textup); the mapping $f$ in the same coordinates is 
$f(x,y,z) = (x,y)$.

\item\label{It:GoodCusp} If $q$ is a good cusp then $f^{-1}(q)$ consists of 
$d-2$ points, and $X$ is smooth at all of them. At exactly one of them $f$ 
is ramified, and the pair $(X,f)$ near this point is locally isomorphic to 
$X = \{(x,y,z) \in \Delta^3 \mid z^3 + 3xz + 2y = 0\}$, $f(x,y,z) = (x,y)$.

\item\label{It:BadCusp} If $q$ is a bad cusp then $f^{-1}(q)$ consists of 
$d-1$ points. All of them except one are smooth, the remaining point is 
singular, locally isomorphic to $\{(x,y,z) \in \Delta^3 \mid z^2=x^3+y^2\}$ 
\textup(Du Val's $A_2$\textup); the mapping $f$ in the same coordinates is 
$f(x,y,z) = (x,y)$.
 \end{enumerate}
 \end{proposition}

 \begin{proof}
A direct computation shows that the branch loci $B$ of the local models 
described above are as follows: $B = \{(x,0) \mid x \in \Delta\}$ for case 
\ref{It:Smooth}, $B = \{(x,y) \in \Delta^2 \mid xy=0\}$ for the cases 
\ref{It:GoodNode} and \ref{It:BadNode}, and $B = \{(x,y) \in \Delta^2 \mid 
x^3+y^2=0\}$ for the cases \ref{It:GoodCusp} and \ref{It:BadCusp}. Thus, in 
all the cases the branch loci are locally biholomorphic to the curve 
$D$ in a neighbourhood of $q$.

By the uniqueness statement of the Grauert--Remmert theorem it is enough to 
check that the monodromy action of $\pi_1(\PP^2 \setminus D)$ 
on the fiber $f_0^{-1}(q) \subset X$ is isomorphic to the action of 
$\pi_1(\Delta^2 \setminus B)$ on the fiber over $0$ in the local model. If 
$a \in f^{-1}(q)$ is a smooth point such that $f$ has the normal form $f(z) 
= z^k$ in it, then the monodromy action is a cycle of length $k$; this 
proves the statement in cases \ref{It:Smooth} and \ref{It:GoodNode}. In 
case \ref{It:BadNode} the branch locus is a union of two lines; the group 
$\pi_1(\Delta^2 \setminus B)$ is generated by two loops circling around 
these lines; apparently, they both act by the same transposition of the two 
preimages of the origin. 

To cover cases \ref{It:GoodCusp} and \ref{It:BadCusp} notice that 
$\pi_1(\Delta^2 \setminus B) = B_3$ in both of them. Any homomorphism $B_3 
\to S_3$ maps the standard generators of the group into permutations $U$ 
and $V$ satisfying $UVU = VUV$; it is easy to see that such $U$ and $V$ 
must be transpositions, either equal or non-commuting. So, there exist only 
two actions, up to conjugation, of $B_3$ on a set of three elements; this 
proves the statement in cases \ref{It:GoodCusp} and \ref{It:BadCusp}.
 \end{proof}

Suppose that $C\subset \PP^2$ is a nodal projective curve that is general 
enough in the sense of Definition~\ref{def:general_curve} and that 
$L\subset(\PP^2)^*$ is a line transversal to $C^*$.

Suppose in addition that $\pi\colon C'\to L$, where $C'$ is a smooth 
projective curve, is a holomorphic mapping with simple ramification and 
with branch locus $L\cap C^*$. Choose a point $t_0\in L\setminus C^*$.

 \begin{corollary}\label{cor:factors=>exists_X}
If the fiber monodromy homomorphism $\pi_1(L\setminus 
C^*,t_0)\to\Perm(\pi^{-1}(t_0))$ factors through the homomorphism 
$\pi_1(L\setminus C^*,t_0)\to \pi_1(\PP^2\setminus C^*,t_0)$ induced by the 
embedding $L\hookrightarrow (\PP^2)^*$, then there exists a unique pair 
$(X,f)$, where $X$ is a normal projective surface and $f\colon X\to 
(\PP^2)^*$ is a finite holomorphic mapping such that 
$\left.f\right|_{f^{-1}(L)}=\pi$.

Over each node of $C$, there is at most one singular point of $X$, and it 
must be a Du~Val $A_1$-singularity. Over each cusp of $C^*$, there is at 
most one singular point of $X$, and it must be a Du~Val $A_2$-singularity. 
The other points of $X$ are smooth.
 \end{corollary}

 \begin{proof}
Since the monodromy $\pi_1(L\setminus C^*,t_0)\to\Perm(\pi^{-1}(t_0))$ 
factors through the homomorphism $\pi_1(L\setminus C^*,t_0)\to 
\pi_1(\PP^2\setminus C^*,t_0)$, there exists a topological covering $X_0\to 
(\PP^2)^*\setminus C^*$ of degree $\deg\pi$ extending the covering 
$\pi^{-1}(L\setminus C^*)\to L\setminus C^*$. The restriction of $f_0$ to 
the line transversal to $C^*$ has simple ramification, so $f_0$ is a 
generic covering by Proposition \ref{prop:same_mono}. The 
rest follows from Proposition~\ref{not.smooth.over.bad} with $D = 
C^*$. 
 \end{proof}

\section{Smooth surfaces simply ramified over duals to general nodal 
curves}\label{sec:PL}

In this section we prove the $(\mathrm b) \Rightarrow 
(\mathrm a)$ implication in Theorem~\ref{th.main}.

Throughout this section we will be working in the following setting. 
$C\subset\PP^2$ is a nodal curve that is general enough in the sense of 
Definition~\ref{def:general_curve}, $C^*\subset(\PP^2)^*$ is the dual 
curve, $X$ is a smooth projective surface, $f\colon X\to(\PP^2)^*$ is a 
finite holomorphic mapping with simple ramification and branch locus~$C^*$. 
Recall that the smooth surface $X_C$ and the mapping $f_C\colon 
X_C\to(\PP^2)^*$ were defined by equations~\eqref{def_of_X_C} 
and~\eqref{def:f_C}, respectively.

For the point $p \in \PP^2$ denote by $\pr_p\colon C \to p^\perp = \PP^1$ 
the projection of $C$ from $p$. Also denote by $\nu\colon\hatC\to C$ the 
normalization mapping. Suppose that $\Phi\colon X\to X_C$ is an isomorphism 
such that $f_C\circ\Phi=f$. 
 \begin{proposition}\label{final_step}
Let $p$ be such that the composition $\pr_p\colon C \to p^\perp$ has simple 
ramification. Then for $C' \mathrel{{:}{=}} f^{-1}(p^\perp)$ there exists 
an isomorphism $\ph\colon \hatC\to C'$ such that $f \circ 
\ph=\pr_p\circ\nu$.
 \end{proposition}

\begin{proof}
The required isomorphism $\ph$ is just the restriction of $\Phi^{-1}$
to $f_C^{-1}(p^\perp)$.
\end{proof}

 \begin{proposition}\label{Chisini_argument}
The $(\mathrm b) \Rightarrow (\mathrm a)$ implication in 
Theorem~\ref{th.main} holds in each of the following cases.

\textup{(1)} $\deg C\ge7$;

\textup{(2)} $\deg C=6$ and $C$ has at least one node;

\textup{(3)} $\deg C=5$ and $C$ has at least three nodes.
 \end{proposition}

 \begin{proof}
Theorem 10 from~\cite{Kulikov1999} says that if $C$ satisfies the 
hypothesis of the proposition, then a finite mapping $f\colon 
X\to(\PP^2)^*$ with simple ramification over $C^*$ is unique. Thus, the 
mapping $f\colon X\to (\PP^2)^*$ is equivalent to $f_C\colon 
X_C\to(\PP^2)^*$, and Proposition~\ref{final_step} applies. 
 \end{proof}

The rest of this section is devoted to the proof of the $(\mathrm b)
\Rightarrow (\mathrm a)$ implication in Theorem~\ref{th.main} in the
remaining cases. This proof is independent of the results of
\cite{Kulikov1999} and works for all the cases in
Theorem~\ref{th.main}, including those covered by
Proposition~\ref{Chisini_argument}.

It follows from the hypotheses on $C$ that the dual curve 
$C^*\subset(\PP^2)^*$ has no inflexion points, all the singularities of 
$C^*$ are nodes and standard cusps, and no line is tangent to $C^*$ at 
three points (bitangents of $C^*$ correspond to nodes of~$C$).

 \begin{proposition}\label{smooth_f^-1}
Suppose that $L\subset(\PP^2)^*$ is a line such that the point $L^\perp$ 
does not lie on $C$. Then the curve $f^{-1}(L)\subset X$ is smooth.
 \end{proposition}

 \begin{proof}
By virtue of projective duality one has $C=(C^*)^*$. Hence, for a line $L 
\subset (\PP^2)^*$, one has $L^\perp \in C$ if and only if $L$ either is 
tangent to $C^*$ at a smooth point, or is tangent to a branch of $C^*$ at a 
node (we call such $L$ a tangent to $C$ at the node), or is the (limiting) 
tangent to $C$ at a cusp. In the third case $L^\perp$ is an inflexion point 
of $C$, in the second case it lies on a double tangent to $C$, and in the 
first case neither takes place. Now the result follows immediately from 
cases \ref{It:Smooth}, \ref{It:GoodCusp} of 
Proposition~\ref{not.smooth.over.bad}, where one puts $D=C^*$.
 \end{proof}

 \begin{proposition}\label{singular_f^-1}
Suppose that a line $L\subset(\PP^2)^*$ is tangent to $C^*$ (at a smooth 
point, a node, or a cusp). Then the curve $f^{-1}(L)$ has singular points 
only over the points of tangency; moreover, there is exactly one singular 
point of $f^{-1}(L)$ over each point of tangency, and this singular point 
is a node. 
 \end{proposition}

 \begin{proof}
Since the curve $(C^*)^*=C$ has no cusps, the curve $C^*$ has no inflexion 
points; now everything follows from Proposition~\ref{not.smooth.over.bad}. 
 \end{proof}

Let the line $L\subset(\PP^2)^*$, where $L^\perp\notin C$, vary. 
Proposition~\ref{smooth_f^-1} shows that the curve $f^{-1}(L)$ is smooth 
for such $L$, so variation of $L$ induces a monodromy action of 
$\pi_1(\PP^2 \setminus C)$ on $H^1(f^{-1}(L),\Z)$. More formally, consider 
the incidence variety
 \begin{equation}\label{inc.var} 
I=\{(p,x)\in \PP^2\times X\colon f(x)\in p^\perp\} 
 \end{equation} 
and denote by $\pr_1: I \to \PP^2$ and $\pr_2: I \to X$ the projections. 
For each $p\in\PP^2\setminus C$ the fiber $\pr_1^{-1}(p)$ is isomorphic to 
$f^{-1}(p^\perp)$. It follows from Proposition~\ref{smooth_f^-1} that the 
derivative of $\pr_1$ has maximal rank everywhere on 
$\pr_1^{-1}(\PP^2\setminus C)$, whence $\pr_1$ restricts to a locally 
trivial fibration over the preimage of $\PP^2\setminus C$. So, for a point 
$q \in \PP^2\setminus C$ the fundamental group $\pi_1(\PP^2\setminus C,q)$ 
acts on $H^1(f^{-1}(q^\perp),\Z)$; this is the monodromy action in 
question.

 \begin{proposition}\label{finite_monodromy}
The image of the monodromy action of $\pi_1(\PP^2 \setminus C,q)$ on
$H^1(f^{-1}(q^\perp),\Z)$ is cyclic of order dividing $d=\deg C$.
 \end{proposition}

 \begin{proof}
Since $C$ is a nodal projective curve, the group $\pi_1(\PP^2\setminus C)$ 
is abelian (see~\cite{Del}), whence it is isomorphic to $H_1(\PP^2\setminus 
C,\Z)$; the latter is obviously cyclic of order~$d$.
 \end{proof}

Recall the main result of local Picard--Lefschetz theory for Lefschetz 
pencils with one-dimensional fiber (\cite[Expos\'e XIV, 3.2.5]{SGA7.2} or 
\cite[\S\S\,5 and~6]{Lamotke}).

 \begin{proposition}\label{PL}
Suppose that \X is a $2$-dimensional complex manifold and $p\colon 
\X\to\Delta$ is a proper surjective holomorphic mapping with the following 
properties.

\textup{(i)} Over $\Delta^*$, the mapping $p$ has no critical points.

\textup{(ii)} In $p^{-1}(0)$, the mapping $p$ has only one critical 
point~$w$ and the Hessian of $p$ at $w$ is non-degenerate.

\textup{(iii)} All fibers of $p$ are connected.

Fix a point $z_0\in \Delta^*$ and put $C=p^{-1}(z_0)$ \textup(in view 
of~\textup{(i)}, $C$ is a compact Riemann surface\textup). Then

\textup{(a)} The curve $C_0=p^{-1}(0)$ is smooth everywhere except for a 
node at~$w$.

\textup{(b)} The curve $C$ contains an embedded circle~$S$ such that $C_0$ 
is homeomorphic to the quotient space $C/S$.

\textup{(c)} The monodromy operator on $H^1(C,\Z)$ corresponding to the 
generator of $\pi_1(\Delta^*)$ is defined by the formula
 \begin{equation}\label{eq:PL}
x\mapsto x-(x,c)c,
 \end{equation}
where $c\in H^1(C,\Z)$ is the Poincar\'e dual of the fundamental class of 
$S$.
 \end{proposition}
(The circle $S$, as well as its Poincar\'e dual cohomology class, is called 
\emph{vanishing cycle}.)

Put $C_0=f^{-1}(\ell)$, where $\ell\subset (\PP^2)^*$ is a general line, 
and denote the genus of $C_0$ by~$g$.

 \begin{proposition}\label{tangent:generalities}
Suppose that $L\subset (\PP^2)^*$ is a line such that $L^\perp \in C$. Then 
the curve $f^{-1}(L)$ is connected, and $L$ can be tangent to $C^*$ only at 
one or two points.
 \end{proposition}

 \begin{proof}
The connectedness of $f^{-1}(L)$ (for arbitrary $L$) follows 
from~\cite[Proposition~1]{FH}. A line $L$ cannot be tangent to $C^*$ at 
more than two points since all the singularities of the curve $C=(C^*)^*$ 
are nodes.
 \end{proof}

 \begin{proposition}\label{f(-1)(tangent)}
If a line $L\subset (\PP^2)^*$ is tangent to $C^*$ at one point, then the 
curve $f^{-1}(L)$ is homeomorphic to the quotient space $C_0/S$, where 
$S\subset C_0$ is homeomorphic to the circle and $S$ is homologous to zero 
in $C_0$.
 \end{proposition}

 \begin{proof}
Take a point $t \in L \setminus C^*$, so that $L^\perp \in t^\perp \subset 
\PP^2$, and $t^\perp$ is transversal to $C$.  Then consider the incidence 
variety
\begin{equation}\label{eq:Lefschetz_pencil}
\tilde X=\{(x, a)\in X\times t^\perp\colon f(x)\in a^\perp\}
\end{equation}
(the surface $\tilde X$ is just the blow-up of $X$ at $f^{-1}(t)$). The 
existence of $S$ follows now from Proposition~\ref{PL} applied to the 
natural projection $\tilde X \to t^\perp$ restricted on the preimage of a 
small disk $\Delta \subset t^\perp$ centered at $L^\perp$. 

We are left only to check that the vanishing cycle is zero-homologous. 
Choosing $z_0\in \Delta\setminus \{L^\perp\}$ and putting $C_0=f^{-1}(L)$, 
observe that $\Delta \setminus L^\perp \subset \PP^2 \setminus C$, so the 
monodromy representation
 \begin{equation}\label{local_mono}
\pi_1(\Delta\setminus \{L^\perp\})\to \Aut(H^1(C_0,\Z))
 \end{equation}
factors through the monodromy representation 
 \begin{equation}\label{global_mono}
\pi_1(\PP^2\setminus C)\to \Aut(H^1(C_0,\Z)).
 \end{equation}
The image of the homomorphism~\eqref{global_mono} is finite by 
Proposition~\ref{finite_monodromy}, whence the image of the generator of 
$\pi_1(\Delta\setminus\{L^\perp\})$ under the 
homomorphism~\eqref{local_mono} has finite order~$k$. The $k$th power of 
the monodromy operator \eqref{eq:PL} is 
 \[
x\mapsto  x - k\cdot(x,c)c,
 \] 
which is identity if and only if $(c,x)=0$ for all $x\in
H^1(C_0,\Z)$. So, $c=0$ and the curve $S$ is homologous to zero.
 \end{proof}

 \begin{proposition}\label{unique_section}
Suppose that $\deg C>2$ and  $L\subset(\PP^2)^*$ is a line tangent to $C^*$ 
at exactly one point~$q$. Then only the following two cases are possible:

\textup{(a)} $f^{-1}(L)=Y_1\cup Y_2$, where $Y_1$ and $Y_2$ are smooth 
projective curves intersecting transversally at a point lying 
over~$q$, where the restriction $\left.f\right|_{Y_1}\colon Y_1\to L$ is an
isomorphism, the restriction $\left.f\right|_{Y_2}\colon Y_2\to L$ has
degree~$>1$, and $(Y_1,Y_1)=0$;

\textup{(b)} $C$ is a smooth cubic curve and $f$ is equivalent to a
projection of the Veronese surface $v_2(\PP^2)\subset\PP^4$ \textup(in
particular, $\deg f=4$\textup). In this case both $Y_1$ and $Y_2$ is a
smooth rational curve, and each of the restrictios
$\left.f\right|_{Y_1}\colon Y_j\to L$ has degree~$2$.
 \end{proposition}

 \begin{proof}
Propositions~\ref{f(-1)(tangent)} and~\ref{tangent:generalities} show
that $f^{-1}(L)=Y$ is a connected curve homeomorphic to the quotient
space $Y'/S$, where $Y'$ is a sphere with handles, and $S\subset Y'$
is a zero-homologous circle. Then $Y'/S$ is homeomorphic to the wedge
sum of two spheres with handles, so the curve $Y$ has exactly two
components.

Since the divisor $Y=Y_1+Y_2$ is the inverse image of the line~$L\subset 
(\PP^2)^*$ and $f(Y_1)=f(Y_2)=L$ (which follows from the finiteness of the 
morphism~$f$), one has, for $j=1$ or~$2$,
 \begin{equation}\label{eq:pos}
(Y_j,Y_1+Y_2)=\deg(\left.f\right|_{Y_j})>0.
 \end{equation}
The curves $Y_1$ and $Y_2$ intersect transversally at one point, whence 
$(Y_1,Y_2)=1$ and therefore
 \begin{equation}\label{both_non-neg}
(Y_1,Y_1)\ge0,\quad (Y_2,Y_2)\ge0.
 \end{equation}

Denote by $V\subset H^2(X,\R)$ the subspace generated by fundamental 
classes of $Y_1$ and $Y_2$. Consider two cases.

\emph{Case 1}: $\dim V=1$. In this case the 
classes of $Y_1$ and $Y_2$ are proportional: $Y_2 = rY_1$.

Observe that the number $r$ must be positive because $(Y_1,f^{-1}L)>0$, 
$(Y_2,f^{-1}L)>0$ for a general line $L\subset(\PP^2)^*$. Since
 \[
1=(Y_1,Y_2)=r(Y_1,Y_1)=r^{-1}(Y_2,Y_2),
 \]
and both $(Y_1,Y_1)$ and $(Y_2,Y_2)$ are integers, one has $r=1$, so
that $Y_1 = Y_2$ in cohomology, and $(Y_1,Y_1)=(Y_2,Y_2)=1$. Since
homology classes of $C$ and $C'$ are equal and $C+C'=f^{-1}(L)$ is an
ample divisor, the divisors $C$ and $C'$ are also ample. Now $\deg
f=(Y_1+Y_2,Y_1+Y_2)=4$. Since fundamental classes of the curves $Y_1$
and $Y_2$ coincide, their genera are the same; denote this number by
$g$.

 \begin{lemma}\label{2g=g}
The curves $Y_1$ and $f^{-1}(L)$ have the same genus.    
 \end{lemma}

 \begin{proof}
Observe that the surface $\tilde X$ defined by the 
equation~\eqref{eq:Lefschetz_pencil} with the natural projection $q\colon 
\tilde X\to t^\perp$ is a Lefschetz pencil. Now a standard argument (see, 
for example, \cite[Expos\'e XVIII, Theorems 5.6.8
  and~5.6.2]{SGA7.2}) shows that image of the natural injection
$H^1(X,\C)\hookrightarrow H^1(f^{-1}(L),\C)$, where the line
$L\subset(\PP^2)^*$ is transversal to $C^*$ and passes through~$t$,
coincides with the invariant subspace
\[
H^1(f^{-1}(L))^{\pi_1(L\setminus C^*)}.
\]
Put $f^{-1}(L)=H$; since the monodromy group is generated by the 
operators~\eqref{local_mono} and all the vanishing cycles~$c$ are zero in 
view of Proposition~\ref{f(-1)(tangent)}, we conclude that the restriction 
$H^1(X,\C)\hookrightarrow H^1(H,\C)$ is an isomorphism, whence the 
homomorphism $H^1(X,\Ooo_X)\to H^1(H,\Ooo_H)$ from the exact sequence
\[
0\to\Ooo_X(-H)\to\Ooo_X\to\Ooo_H\to 0
\]
is also an isomorphism; the Kodairra vanishing theorem and Serre's duality 
show then that this is equivalent to the equation $\dim \lmod K_X\rmod = 
\dim \lmod K_X+H\rmod$; since $\dim \lmod H\rmod > 0$, this is equivalent 
to the relation $\lmod K_X+H\rmod=\varnothing$ (cf.~\cite[Proposition 
6.1]{Lvovski}).

Since $H=Y_1+Y_2$ and $Y_2$ is an effective divisor, it follows now that 
the linear system $\lmod K_X+Y_1\rmod$ is also empty; since the divisor 
$Y_1$ is ample, the above argument in the reverse order shows that 
$H^1(X,\Ooo_X) \cong H^1(Y_1,\Ooo_{Y_1})$. Thus, genera of the curves 
$f^{-1}(L)=H$ and $Y_1$ are both equal to $\dim H^1(X,\Ooo_X)$. 
 \end{proof}

If a line $L'\subset(\PP^2)^*$ is transversal to $C^*$, then, according to 
Proposition~\ref{PL}b, $Y_1\cup Y_2$ is homeomorphic to $Y'=f^{-1}(L')$ 
with a circle contracted to a point. Hence, the genus of $Y'$ equals $2g$. 
Since $2g=g$ by Lemma~\ref{2g=g}, we infer that $g=0$. Thus, $Y_1$ and 
$Y_2$ are smooth rational ample curves with self-intersection indices~$1$ 
on the smooth projective surface~$X$. It is well known (see, for example, 
\cite[Proposition 2.3]{Lvovski2}) that if a smooth rational curve $Y$ on a 
smooth projective surface $X$ is ample and $(Y,Y)=1$, then there exists an 
isomorphism from $X$ to $\PP^2$ mapping $Y$ to a line. Thus, $X\cong \PP^2$ 
and $\Ooo_X(f^{-1}(L))\cong \Ooo_{\PP^2}(2)$, so the mapping $f\colon 
X\to\PP^2$ is isomorphic to a projection $\pr_\Lambda\colon 
v_2(\PP^2)\to\PP^2$, where $v_2(\PP^2)\subset \PP^5$ is the quadratic 
Veronese surface and $\Lambda\subset\PP^5$ is a $2$-plane disjoint from 
$v_2(\PP^2)$. Since the curve $C^*$ is the branch locus of the projection 
$\pr_\Lambda$, its dual curve $C\subset\PP^2$ coincides with the 
intersection $(v_2(\PP^2))^*\cap\Lambda^\perp$ (cf. the proof of 
Proposition~\ref{prop:example} below), which is clearly a smooth plane 
cubic since $(v_2(\PP^2))^*$ is the variety of degenerate symmetric 
$3\times3$-matrices.

\emph{Case 2}: $\dim V = 2$. By the Hodge index theorem, 
 \[
\begin{vmatrix}
(Y_1,Y_1)&(Y_1,Y_2)\\
(Y_1,Y_2)&(Y_2,Y_2)
\end{vmatrix}
<0,
 \]
whence $(Y_1,Y_1)(Y_2,Y_2)<1$. So at least one of the factors must be
zero; assume without loss of generality that $(Y_1,Y_1)=0$, then
 \[ 
\deg(\left.f\right|_{Y_1})=(Y_1,Y_1+Y_2)=1
 \]
Thus, the restriction of $f$ to $Y_1$ is an isomorphism from $Y_1$ to the 
line $L$. Let us prove that degree of the restriction 
$\left.f\right|_{Y_2}\colon Y_2\to f(Y_2)$ is greater than one. Indeed, $L$ 
is tangent to $C^*$ at only one point, so $Y$ has only two components. 
Thus, if $Y_2$ maps with degree~$1$ onto~$L$, then
 \[ 
(Y_2,Y_2)=(Y_2,Y_1+Y_2)-(Y_2,Y_1)=1-1=0
 \]
and $\deg f=(Y_1+Y_1,Y_2+Y_2)=2(Y_1,Y_2)=2$, so all nodes and/or cusps of 
$C^*$ must be bad in the sense of Definition~\ref{def:bad.cusp}, which 
contradicts the smoothness of $X$ in view of 
Proposition~\ref{not.smooth.over.bad}. Thus, $C^*$ must be smooth. Since 
its dual curve $(C^*)^*=C$ has no cusps, the smooth curve~$C^*$ has no 
inflexion points. It is well known (and follows, for example, from the 
Pl\"ucker formulas~\cite[Section 9.1, Theorem 1]{BK}) that the only smooth 
plane curve without inflexions is the conic, whence $C$ is a conic, which 
contradicts the hypothesis.
 \end{proof}

 \begin{note}
The trick with Hodge index theorem is essentially contained in Van de Ven's 
paper~\cite{VandeVen} (see the proof of Theorem~I). A similar argument 
allows one to give a proof of the main result of the paper~\cite{Zak} that 
is valid in arbitrary characteristic (see~\cite{Lvovski2}).
 \end{note}

 \begin{proposition}\label{X_C=X}
Suppose that $C\subset\PP^2$ is a smooth or nodal curve of degree~$>2$ that 
is general enough in the sense of Definition~\ref{def:general_curve}. Let 
also $X$ be a smooth projective surface and $f\colon X\to(\PP^2)^*$ be a 
finite morphism with branch locus $C^*$ and such that the ramification is 
simple. If $\deg C=3$, assume in addition that $C$ has a node or $\deg f\ne 
4$. Then there exists an isomorphism $\Phi\colon X\to X_C$ such that 
$f_C\circ\Phi=f$, where $X_C$ and $f_C$ are defined by 
equations~\eqref{def_of_X_C} and~\eqref{def:f_C}, respectively.
 \end{proposition}

We know two proofs of this proposition, one ``topological'' and one 
algebraic geometric. For the expositon in this paper, we have chosen the 
latter since its rigorous version appears to be shorter.

 \begin{proof}
As before, denote the normalization of $C$ by $\nu\colon \hatC\to C$. Let 
$\gamma\colon \hatC\to C^*$ be the Gauss mapping, which attaches to a point 
$x\in C$ the tangent to $C^*$ at $\nu(x)$, where by tangent we mean the 
limiting tangent if $\nu(x)$ is a cusp of $C^*$ and the limiting tangent at 
the branch corresponding to $x$ if $\nu(x)$ is a node of~$C^*$.

If follows from projective duality that the definition of the surface
$X_C$ (see~\eqref{def_of_X_C}) may be rewritten as
 \[
X_C=\{(x,t)\in \hatC\times(\PP^2)^*\mid t\in \gamma(x)^\perp\}. 
 \]
Now put
 \[
\Zcal=\{(x,y)\in \hatC\times X\mid f(y)\in \gamma(x)^\perp\}.
 \]

  \begin{lemma}\label{lemma:component}
There exists a component $\Zcal_1\subset\Zcal$ consisting of
components of fibers that have intersection index $1$ with
$D=f^{-1}(L)$ \textup(or, equivalently, project isomorphically onto
tangents to~$C^*$; in the proof of Proposition~\ref{unique_section}
these components were called~$Y_1$\textup).
  \end{lemma}

  \begin{proof}[Proof of the lemma]
We will be using the language of schemes.

Denote the natural projection $\Zcal\to\hatC$ by $\pi_1$ and the morphism 
$(x,y)\mapsto f(y)$ by $\pi_2 \colon\Zcal\to(\PP^2)^*$. Pick a line 
$L\subset(\PP^2)^*$ transversal to $C^*$ and denote by $D\subset\Zcal$ the 
closed subscheme $\Pi_2^{-1}L$. Proposition~\ref{unique_section} shows that 
for almost all (in the sense of Zariski topology) closed points 
$x\in\hatC$, the fiber $\pi_1^{-1}(x)$ consists of two 
components $Y_1$ and $Y_2$ and $D\cap \pi_1^{-1}(x)$ is a reduced 
scheme of length $\deg C^*$; moreover, one of the points of $D\cap 
\pi_1^{-1}(x)$ lies in $Y_1$ and the remaining $\deg C^*-1$ 
lie in $Y_2$. Let $K$ stand for the function field of \hatC and 
$\bar K$ for its algebraic closure; put $\eta=\Spec K$, $\bar\eta=\Spec\bar 
K$, and let $\Zcal_\eta$ and $\Zcal_{\bar\eta}$ be the generic and generic 
geometric fibers of $\pi_1$ respectively. Then it is apparently evident 
that $\Zcal_{\bar\eta}$ consists of two components and the pullback of $D$ 
is also a reduced scheme of length $\deg C^*$ such that one of its points 
lies on one of the components of $\Zcal_{\bar\eta}$, while the remaining 
$\deg C^*-1$ points lie on the other component. To prove this assertion 
rigorously, we invoke EGAIV.

To wit, consider the coherent sheaf $\F=\Ooo_\Zcal\oplus\Ooo_D$ on
\Zcal.  The primary type (``type primaire'') \cite[Remarque
  9.8.9]{EGAIV-3} (roughly speaking, it is the collection of lengths
of fibers of \F at the generic points of components of its support and
at the generic points of various intersections of the closures of
these points) of pullbacks of \F to fibers over almost all (in the
sense of Zariski topology) closed points is the same; since the set of
points with the property ``the pullback of \F to the geometric fiber
over the point in question has a given \emph{type primaire}'' is
constructible due to~\cite[9.8.9.1]{EGAIV-3}, it follows that
\emph{type primaire} of the pullback of \F to $\Zcal_{\bar\eta}$ is
the same as that of its pullbacks to almost all fibers over closed
points, whence the assertion.

Now observe that the Galois group $G=\mathrm{Gal}(\overline K/K)$ acts on the 
set of the two components of $\Zcal_{\bar\eta}$. Since the pullback of $D$ 
to $\Zcal_{\bar\eta}$ is $K$-rational, it is $G$-invariant. On the other 
hand, one of these component contains only one point of $D$, while the 
other contains $\deg C^*-1>1$ such points (if $\deg C^*-1=1$, then $C^*$ 
and $C$ are conics, which contradicts the hypothesis). Thus, both these 
components are $G$-invariant, whence $\Zcal_\eta$ contains a component 
containing only one point from the pullback of~$D$. This proves the lemma.
  \end{proof}

Returning to the proof of the proposition, observe that
self-intersection index of almost all of the fibers of the projection
$\Zcal_1\to C^*$ as curves on $X$ is zero due to
Proposition~\ref{unique_section}, so they are disjoint. Thus, the
natural mapping $\Phi_1\colon\Zcal_1\to X$ (induced by projection to
the second factor) is generically one to one; since $X$ is smooth,
Zariski's main theorem implies that $\Phi_1$ is an isomorphism. On the
other hand, if we define the morphism $\Phi_2\colon \Zcal_1\to X_C$ by
the formula~$(x,y)\mapsto (x,f(y))$, then it is easy to see that
$\Phi_2$ is also generically bijective, whence isomorphic. Now we may
put $\Phi=\Phi_2\circ\Phi_1^{-1}$.
 \end{proof}

Now the $(\mathrm b) \Rightarrow (\mathrm a)$ implication in 
Theorem~\ref{th.main} follows immediately from Propositions~\ref{X_C=X} 
and~\ref{final_step}.

 \begin{note}\label{chisini:example}
If $\deg C=3$ and $\deg f=4$, Theorem~\ref{th.main} does not
hold. Indeed, if $f\colon v_2(\PP^2)\to\PP^2$ is a general projection
of the quadratic Veronese surface, then the branch locus of $f$ is the
curve $C^*\subset \PP^2$ that is dual to a smooth cubic~$C$. If
$L\subset \PP^2$ is a general line, then the restriction
$\pi=\left.f\right|_{f^{-1}(L)}\colon f^{-1}(L)\to L$ is a generic
covering of degree $4$ ramified over $L\cap C^*$, that is, over the
branch locus of the projection $\pr_{L^\perp}\colon C\to L$. By the
very construction, the mapping $\pi$ can be extended to the mapping
$f\colon v_2(\PP^2)\to\PP^2$ ramified over~$C^*$, but $\deg f=4\ne3$,
so $f$ is not equivalent to a projection of the plane cubic.
 \end{note} 

The following corollary of Theorem \ref{th.main} is essentially 
its ``topological'' reformulation:

 \begin{corollary}\label{th.main2}
Suppose that $C\subset \PP^2$ is a nodal curve of degree~$>2$ that is 
general enough in the sense of Definition~\ref{def:general_curve}. Suppose 
that a point $p\in\PP^2\setminus C$ is such that the composition $\pr 
_p\circ\nu\colon\hatC\to\PP^1$, where $\nu\colon\hatC\to C$ is the 
normalization and $\pr _p$ is the projection from $p$, has simple 
ramification. Denote by $L\subset(\PP^2)^*$ the line in the dual plane 
corresponding to the point $p\in\PP^2$, and suppose that $C'$ is a smooth 
projective curve and $\pi\colon C'\to L$ is a holomorphic mapping with 
simple ramification and such that the branch locus of $\pi$ coincides with 
$L\cap C^*$. If $\deg C=3$, assume in addition that $C$ has a node or $\deg 
f\ne4$. Then the following two conditions are equivalent.

\textup{(a)} There exists an isomorphism $\ph\colon C'\to\hat C$ such that
$(\pr _p\circ\nu)\circ \ph=\pi$.

\textup{(b)} The covering $\pi^{-1}(L\setminus C^*)\to L\setminus C^*$
can be extended to a covering $X_0\to(\PP^2)^*\setminus C^*$ with
respect to which all nodes and cusps of the curve~$C^*$ are good in the
sense of Definitions~\ref{def:good.nodes} and~\ref{def:bad.cusp}.
 \end{corollary}

 \begin{proof}
Immediate from Theorem~\ref{th.main} and 
Proposition~\ref{not.smooth.over.bad}.
 \end{proof}

Our proof of Theorem~\ref{th.main} implies the following corollary:

 \begin{corollary}\label{chisini_duals}
Suppose that $C\subset\PP^2$ is a nodal curve that is general enough
in the sense of Definition~\ref{def:general_curve}. If $C$ is not a
smooth cubic, then any two generic coverings $X\to(\PP^2)^*$, ramified
over $C^*$ and with smooth $X$, are equivalent.
 \end{corollary}

If $\deg C \ge 7$, or $\deg C = 6$ and $C$ has at least one node, or
$\deg C = 5$ and $C$ has at least three nodes, then this corollary is
implied by \cite[Theorem~10]{Kulikov1999}.

It is well known (see, for example,~\cite{Kulikov1999}) that this assertion 
is wrong if $C$ is a smooth cubic. A~relevant counterexample is
  essentially contained in our proof of
  Proposition~\ref{unique_section} (Case~1).

\section{The smoothness condition is (sometimes) 
essential}\label{sec:example}

In this section we show that the smoothness condition in Theorem 
\ref{th.main} is not trivial, at least for the case of smooth curves of 
degree $d=3$ or $4$ and mappings of degree~$d$.

To wit, we prove the following.

 \begin{proposition}\label{prop:example}
Suppose that $d=3$ or $4$. Then there exist a smooth curve $C\subset\PP^2$, 
$\deg C=d$, a line $p^\perp\subset(\PP^2)^*$ (where $p \in \PP^2$ is a 
point) transversal to~$C$, a smooth projective curve $C'$, and a 
holomorphic mapping $\pi\colon C'\to p^\perp$ with simple ramification such 
that $\deg\pi=d$ and $\pi$ is ramified exactly over $C^*\cap p^\perp$ with 
the following properties.

\textup{(1)} For a point $t_0\in p^\perp \setminus C^*$, the fiber monodromy
homomorphism $\pi_1(p^\perp\setminus C^*,t_0)\to \Perm (\pi^{-1}(t_0))$ can
be factored through the homomorphism $\pi_1(p^\perp\setminus C^*,t_0)\to
\pi_1((\PP^2)^*\setminus C^*,t_0)$ induced by the embedding
$p^\perp \setminus C^*\hookrightarrow (\PP^2)^*\setminus C^*$.

\textup{(2)} The mapping $\pi\colon C'\to p^\perp$ is \emph{not} equivalent
to the projection $\pr _p\colon C\to p^\perp$. 
\end{proposition}

We will need the following lemma.

  \begin{lemma}\label{self-dual:3,4}
If $d=3$ or $4$, then there exists a surface with isolated singularities 
$X\subset\PP^3$, $\deg X=d$, such that its dual $X^*\subset(\PP^3)^*$ is 
isomorphic to $X$ and the general hyperplane section of $X$ is general 
enough in the sense of Definition~\ref{def:general_curve}.
  \end{lemma}

 \begin{proof}[Proof of the proposition modulo Lemma~\ref{self-dual:3,4}] 
Let $X \subset \PP^3$ be a surface from Lemma~\ref{self-dual:3,4}. For a 
general point $r\in\PP^3$, the projection $\pr_r\colon X\to \PP^2$, where 
$\PP^2$ is a projective plane (incidentally, $\PP=(r^\perp)^*$), is simply 
ramified. If $B\subset\PP^2$ is the branch locus of $\pr_r$, then a line 
$\ell\subset\PP^2$ is tangent to $B$ if and only if the plane $\Pi_r$ 
spanned by $r$ and $\ell$ is tangent to $X$. Thus, the dual curve 
$B^*\subset(\PP^2)^*$ is projectively isomorphic to $X^*\cap r^\perp =C$; 
if $p$ is general, Lemma~\ref{self-dual:3,4} asserts that $C$ is a general 
enough smooth plane curve.

Take $p \in \PP^3$ general and let $C'=X\cap\Pi_p$; denote the restriction 
of $\pr_p$ to~$C'$ by $\pi\colon C'\to L$. We see (taking into account that 
$B=C^*$) that $\pi$ extends to a finite morphism $\pr_p\colon X\to 
(\PP^2)^*$ ramified over $C^*$, where the surface $X$ is singular. Thus, 
$\pi$ satisfies Condition (b) of Theorem \ref{th.main} with the smoothness 
omitted. On the other hand, $\pi$ cannot be equivalent to the projection 
$\pr_r$. Indeed, if such an equivalence existed, then by Proposition 
~\ref{gen(a)=>(b)} the restriction of $\pi$ to the preimage of the 
complement of the branch locus could be extended to a finite morphism 
$X_C\to\PP^2$ ramified over $C^*$ with smooth $X_C$. However, if such an 
extension to a finite mapping $X\to(\PP^2)^*$ ramified over $C^*$ and with 
normal (in particular, smooth) $X$ exists, this extension is unique by the 
theorem of Grauert and Remmert~\cite[Expos\'e~XII, Th\'eor\`eme 5.4]{SGA1} 
we cited before. Since we already have such an extension with singular $X$, 
we arrive at the desired contradiction. 
 \end{proof}

It remains to prove Lemma~\ref{self-dual:3,4}. If $d=3$, we let $X$ be the 
projective closure of surface in $\mathbb A^3$ defined by the equation 
$x_1x_2x_3=1$ (this surface was considered in~\cite{ACT}). A direct 
computation shows that $X$ contains only three singular points (of the type 
$A_2$) and that $X$ is projectively isomorphic to~$X^*$. Since any smooth 
plane cubic is automatically general enough in the sense of 
Definition~\ref{def:general_curve}, the surface $X$ is as required.

To treat the case $d=4$ we need to recall the definition and some 
properties of Kummer surfaces.

Suppose that $J$ is a principally polarized Abelian variety of 
dimension~$2$ that is not the product of two elliptic curves and 
$\Theta\subset J$ is a theta divisor of polarization. Then the complete 
linear system $\lmod 2\Theta\rmod$ is base point free and defines a finite 
morphism $J\to\PP^3$; the image of this morphism is a surface of degree $4$ 
with exactly $16$ singular points of the type $A_1$; the surfaces obtained 
by this construction are called Kummer surfaces (see, for example, 
\cite[Section 10.3]{Dolgachev}). Any Kummer surface $X\subset\PP^3$ is 
projectively isomorphic to its dual $X^*\subset(\PP^3)^*$ (see 
\cite[Theorem 10.3.19]{Dolgachev}). Thus, to finish the proof of 
Lemma~\ref{self-dual:3,4}, it remains to show that there exists a Kummer 
surface such that its general plane section is general enough in the sense 
of Definition~\ref{def:general_curve}; since such sections are plane 
quartics, it suffices to establish that a general plane section has only 
simple inflexion points.

The following lemma and its proof were pointed out to us by Rita Pardini. 

 \begin{lemma}[R. Pardini]\label{Pardini}
Each non-hyperelliptic curve of genus $3$ \textup(smooth and 
projective\textup) is isomorphic to a plane section of an appropriate 
Kummer surface in~$\PP^3$.
 \end{lemma}

 \begin{proof}[Proof of the lemma]
Suppose that $C$ is such a curve, and let $\sigma\colon C_1\to C$ be an 
unramified covering of degree~$2$. Denote by $(J,\Theta)$ the Prym variety 
corresponding to this covering~\cite[Section 12.2]{BiLa}. According to 
\cite[Proposition~12.5a]{BiLa}, the Abel--Prym map embeds $C'$ in $J$ as an 
element of the linear system~$2\Theta$. Multiplication by $-1$ on $J$ 
restricts on $C_1$ to the involution induced by $\sigma$, hence the image 
of $C_1$ via the map given by $\lmod 2\Theta\rmod$ is a plane section of the 
Kummer surface associated to $(J,\Theta)$, and this section is isomorphic 
to~$C$.
 \end{proof}

Now if we pick a smooth plane curve $C_0\subset\PP^2$ of degree $4$ such 
that all its Weierstrass points have index one, or, equivalently, all its 
inflexion points are simple, Lemma~\ref{Pardini} shows that $C_0$ is a 
plane section of a Kummer surface $X\subset\PP^3$; since the property ``all 
the Weierstrass points have index one'' is open, a generic plane section of 
$X$ has only simple inflexion points, which completes the proof of 
Lemma~\ref{self-dual:3,4}.

\bibliographystyle{amsalpha}
\bibliography{proj_new}

\end{document}